\documentclass{daj}

\usepackage{amsmath}
\usepackage{amsthm, amssymb, amsfonts}

	\def\eps{\varepsilon}

	\def\C{\mathbb{C}}
	\newcommand{\R}{\mathbb{R}}
	\def\Z{\mathbb{Z}}
	
	\def\l{\lambda}
	
	\def\t{\theta}

	\def\N{\mathbb{N}}

	\newcommand{\Cont}{\mathrm{Cont}}
	\newcommand{\tlambda}{\widetilde{\lambda}}
	\renewcommand{\Re}{\operatorname{Re}}
	\renewcommand{\Im}{\operatorname{Im}}	
	
	\def\eit{e^{i\theta}}

	\newtheorem{theorem}{Theorem}
	\newtheorem{lemma}[theorem]{Lemma}
	\newtheorem{claim}[theorem]{Claim}
	\newtheorem*{claim*}{Claim}
	\newtheorem{corollary}[theorem]{Corollary}

	\newtheorem{observation}[theorem]{Observation}

	\theoremstyle{definition}
	\newtheorem{definition}[theorem]{Definition}

	\newtheorem{case}{Case}
	
	\theoremstyle{remark}

\dajAUTHORdetails{%
  title = {A Characterization of Polynomials Whose High Powers Have Non-negative Coefficients}, 
  author = {Marcus Michelen, Julian Sahasrabudhe},
  plaintextauthor = {Marcus Michelen and Julian Sahsrabudhe},
  runningtitle = {Eventually non-negative polynomials}, 
}   

\dajEDITORdetails{%
   year={2020},
   number={20},
   received={1 April 2020},   
   published={31 December 2020},  
   doi={10.19086/da.18560},       
}   

\begin{document}

\begin{frontmatter}[classification=text]


\author[marcus]{Marcus Michelen}
\author[julian]{Julian Sahasrabudhe}

\begin{abstract}
Let $f \in \R[z]$ be a polynomial with real coefficients. We say that $f$ is \emph{eventually non-negative} if $f^m$ has non-negative coefficients 
	for all sufficiently large $m \in \N$. In this short paper, we give a classification of all eventually non-negative polynomials. This generalizes a theorem of De Angelis, and proves a conjecture of Bergweiler, Eremenko and Sokal.
\end{abstract}
\end{frontmatter}

\section{Introduction}

In this short paper we study the following basic problem about iterated convolutions of sequences of real numbers. 
\begin{center}
For what sequences $S  = (c_0,\ldots,c_d)$ are all ``high'' convolutions $S \ast S \ast \cdots \ast S$ non-negative?
\end{center}
This is the same as asking,
\begin{center}
For what polynomials $f \in \R[x]$ does $f^m$ have non-negative coefficients for all large $m$?'
\end{center}
In this note we give a classification of these polynomials by showing that two natural necessary conditions are also sufficient.

Polynomials with non-negative coefficients are of particular interest as they enjoy special properties and appear in many combinatorial, physical and probabilistic settings (see, for example, \cite{scott-sokal,BBL,pemantle-survey,LPRS,yang-lee,lee-yang,clt1,clt2}). Thus classifying powers of polynomials for which $f^m$ has non-negative coefficients can be a useful tool for thinking about the space of polynomials with non-negative coefficients and building such polynomials in non-trivial ways.

With this in mind, it perhaps not a surprise that questions of this type reach back to the work of Poincar\'e \cite{poincare}, who studied polynomials $f$ for which there exists a polynomial $p$ so that $pf$ 
has all positive (as opposed to non-negative) coefficients. He gave a full characterization of such polynomials by showing that $f$ satisfies this condition if and only if $f(r) > 0$ for all $r > 0$.  P\'olya \cite{polya1928positive} later proved a multivariate analogue which, in the univariate case, shows $p$ can be taken to be a sufficiently high power of $1 + z$.  This line of results culminated in a pair of papers by Handelman \cite{handelman-85,handelman} who gave necessary and sufficient conditions on the pair $(p,f)$ so that $p(z)^m f(z)$ has non-negative coefficients for some large $m$, provided that $p$ has non-negative coefficients. 
Interestingly, Handelman \cite{handelman1992polynomials} also proved that if a monic polynomial $f$, with $f(1) >0$, is such that $f^m$ has non-negative coefficients for \emph{some} $m$, then it has non-negative coefficients for \emph{all} sufficiently large $m$.

More recently, De Angelis \cite{deAngelis-positivity} studied polynomials $f$ for which $f^m$ has \emph{strictly positive} coefficients for all large $m$. 
This is a particularly interesting property as raising a polynomial to a large power does not change the distribution of the roots (up to multiplicities). Indeed, it is for this reason that this problem was also studied by Bergweiler, Eremenko and Sokal \cite{bes} and was a key step in their work on the root distributions of polynomials with positive coefficients \cite{bergweiler-eremenko}.

For their results, the authors of both \cite{deAngelis-positivity,bergweiler-eremenko} used an interesting notion of positivity: We say that a non-zero polynomial $f$ is \emph{strongly positive} if 
\begin{equation}\label{eq:strong-positive}  f(|z|) > |f(z)|,\end{equation} for all $z \in \C \setminus \R_{\geq 0}$. To see the connection with powers of $f$, we note that a polynomial $f$ is such that $f^m$ has non-negative coefficients then $f$ is strongly positive, as long as $f \not= z^kg(z^{\ell})$, for $\ell \geq 2$. Indeed, if $f^m$ has all non-negative coefficients, for some $m$, then for all $z \in \C \setminus \R_{\geq 0}$ we have
\[ |f(z)|^m = |f^m(z)| \leq f^m(|z|),\]
by the triangle inequality. Additionally, we see that this inequality is strict due to the fact that $f \not= z^kg(z^{\ell})$, for all $\ell$, and therefore $|f(z)| < f(|z|)$.

De Angelis gave the following classification of \emph{eventually positive} polynomials: that is, polynomials for which 
all of the coefficients $a_i$ of $f^m$ are strictly positive for all $i \in \{0,1, \ldots, m\deg(f)\}$ when $m$ is sufficiently large.

\begin{theorem}[\cite{deAngelis-positivity}] \label{th:deAngelis}
	Let $f(z) = a_0 + a_1 z + \cdots + a_{d} z^d$ be a polynomial with real coefficients.
	Then $f$ is eventually positive if and only if $f$ satisfies $|f(z)| < f(|z|)$ for all $z \in \C \setminus \R_{\geq 0}$ and $a_0, a_d, a_1,a_{d-1} > 0$. 
\end{theorem}

In this paper, we classify polynomials for which the coefficients of $f^m$ are \emph{non-negative}, rather than positive. A special case of this classification problem was conjectured to hold by Bergweiler, Eremenko and Sokal \cite{bes}
(see p.\ 11). Moreover, Tan and To \cite{tan-to-2017,tan-to-2019}, who have recently extended Theorem \ref{th:deAngelis} to the multivariate setting, state that
new ideas will be necessary to tackle the case of non-negative coefficients. Let us say that a polynomial $f$ is \emph{eventually non-negative}
if $f^m$ has only non-negative coefficients for all large $m$.

For our classification, we will use the notion of ``strong positivity'' but will additionally need to supply a more sophisticated ``boundary condition'' for the coefficients: for example, we cannot simply assume that $a_1 > 0$, as De Angelis does. To get a feel for what we need, let us consider the simplest \emph{non-example}---when $a_0 = -1$. Obviously, as $m$ grows, the constant term in $f^m$ simply alternates between $1$ and $-1$. This is because there is no interaction with 
any of the other terms and there is no chance that it becomes positive. While this example is trivial, this behavior can occur in a slightly more complicated setting. Indeed, consider the polynomial that starts $f = 1 + z^2  - z^3 + \cdots $. Here, it is not possible 
for the coefficient of $z^3$ to become non-negative, as there is no way for the positive terms $1,z^2$ to ``add up'' and reach $z^3$.
Now, one might hope that such examples could simply be excluded by the strong positivity hypothesis. However, it is not hard to construct examples
that are additionally strongly positive. Indeed, Bergweiler, Eremenko and Sokal observed that for $\eps > 0$ sufficiently small, the family of polynomials $$1 + z^3 + z^4 - \eps z^5 + z^6 + z^7 + z^{10}$$ is strongly positive, but all powers have a negative coefficient of $z^5$,  for this same reason. 

Our ``boundary condition'' simply excludes these ``obvious'' ways of ensuring that $f^m$ has negative coefficients. 

\begin{definition}
Let $f(z) = a_0 + a_1 z + \cdots + a_d z^d$ be a degree $d$ polynomial with real coefficients and $a_0,a_d > 0$. Let $S^+(f) := \{ i : a_i >0\}$ denote the
indices of the positive coefficients and and let $S^{-}(f) := \{ i : a_i <0 \}$ denote the indices of the negative. We say that $f$ has the 
\emph{one-sided positive covering property} if every element in $S^-(f)$ can be written as a sum of elements in $S^+(f)$. That is,
\[ S^+(f) + \cdots + S^+(f) \supseteq S^-(f), \]
where the left-hand-side denotes the $d$-times iterated sumset of $S^+(f)$.
We then define $f$ to have the \emph{positive covering property} if both $f(z)$ and its ``reverse'' $z^df(1/z)$ have the one-sided covering property.
\end{definition}

As a quick check, it is useful to note that De Angelis's ``boundary'' condition, $a_0,a_1,a_d,a_{d-1} >0$, easily implies that $f$ satisfies the 
positive covering property. In this case, we have $0,1 \in S^{+}(f)$ and therefore $\{1,\ldots, d\}$ is a subset of the $d$ fold sum set $S^+(f) + \cdots + S^+(f)$.
With this observation at hand, it will be easy to see that our Theorem~\ref{thm:main} naturally implies De Angelis's Theorem~\ref{th:deAngelis} and thus can be viewed as a generalization.

We now state the main theorem in this paper, which says that the necessary conditions of strong positivity and the positive covering property
are sufficient.

\begin{theorem} \label{thm:main}
Let $f \in \R[x]$ be a non-constant polynomial, then $f$ is eventually non-negative if and only if $f = z^{k}g(z^{\ell})$, where $k, \ell \in \N$ and $g$ is strongly positive and satisfies the positive covering property. 
\end{theorem}

To see why these conditions are necessary, we spell out the details a little more carefully in Lemma~\ref{lem:easy-direction} in Section~\ref{sec:proof}.

The proof that these conditions are sufficient has two main steps. In the first step we use the positive covering property alone to show that the first and last $m^{1-\eps}$ coefficients of $f^m$ must be non-negative. This is achieved by a careful counting argument along with a ``self improving'' trick. In the second step, we show that we can lift this proof to show that all $[0,\delta m]$ coefficients in $f^m$ are non-negative. We do this by first using our positive covering property to obtain control $f$ on a certain region close to the origin and then applying a saddle-point method. We then finish the proof of Theorem~\ref{thm:main} by using the strong positivity property \eqref{eq:strong-positive} and a lemma of De Angelis.

\section{Dealing with $n \in [0,m^{1-\eps}]$} \label{sec:small-coefficients}

Our proof of non-negativity of coefficients in the range $[0,m^{1-\eps}]$ for $\eps \in (0,1)$ is purely combinatorial and relies solely on the positive covering hypothesis.  Our goal will be to prove the following lemma.  For a polynomial $f$ we write $[z^k] f$ to denote the coefficient of $z^k$ in $f$. 

\begin{lemma}\label{lem:poly-many}
	Let $f \in \R[z]$ be a polynomial of degree $d \in \N$ that satisfies the positive covering property.  Then for each $\eps \in (0,1)$, there exists an $M = M(\eps)$ so that for all $m \geq M$, $[z^n] f^m \geq 0$ for all $n \in [0, m^{1-\eps}]$. 
\end{lemma}

To understand the combinatorics at play here, it is useful to look at a basic example first.  Consider the polynomial $f(z) = 1 + z^2 + z^3 - Az^5$ for some $A > 0$, and examine the coefficient of $z^5$ in $f^m$: expanding the polynomial directly shows $[z^5] f^m = \binom{m}{2} - A m$, which is positive for large $m$, even if $A$ is massive.  The idea is that the freedom of choosing more terms of smaller degree---i.e.~$z^2$ and $z^3$ rather than $z^5$ in this example---gives a larger contribution than swapping out some number of small degree terms for a larger degree term, even if it has a large negative coefficient. Our method for proving Lemma \ref{lem:poly-many} will be to push this idea much further. 

Before jumping in, we need an elementary bound on ratios of binomial coefficients.

\begin{lemma}\label{lem:gen-binomial-bound}
	Let $d > 0$ and $\gamma \in (0,1)$.  There exist $A > 0$  and $C_d > 0$ so that for all $a \geq A$ and $b \leq a^{\gamma}$ we have $$\frac{\binom{a}{b}}{\binom{a-i}{b + j}} \leq  C_d a^{-j(1 - \gamma)}$$
	for all $i$ with $|i| \leq d$, $j \in \{1,\ldots,d\}$.
\end{lemma}
\begin{proof} Let $(a)_b = a(a-1)\cdots(a - b+1)$ denote the falling factorial. We have 
$$\frac{\binom{a}{b}}{\binom{a-i}{b + j}} = \frac{(a)_b (b + j)!}{b! (a-i)_{b+j}} \leq 2^{i+j} \frac{(b+j)^j}{a^j} \leq 2^{i+j} \cdot 2^j a^{1 - \gamma} \leq 8^d a^{1-\gamma}$$
	for sufficiently large $a$.  Choosing $C_d = 8^d$ completes the proof.
\end{proof}

Each term in the expansion of $f^m$ contributing to $[z^n] f^m$ can be associated to a partition $\lambda = (1^{\lambda_1},2^{\lambda_j},\ldots,d^{\lambda_d})$ of $n$ in which we choose the monomial $a_jz^j$ a total of $\lambda_j$ times.  Then the contribution associated to a given partition $\lambda$ is equal to 
$$\Cont(\lambda) := \binom{m}{\l_1} \binom{m-\l_1}{\l_2} \cdots \binom{m-\l_1 -\cdots - \l_{d-1}}{\l_d} \prod_i a_i^{\l_i}.$$
 We may thus write \begin{equation}\label{eq:new-anm-cont}
	[z^n] f^m = \sum_{\lambda \vdash n} \Cont(\lambda),
	\end{equation}
	where the sum is over all integer partitions of $n$.
	
	In the next lemma, we shall show that for each $n$ (in an appropriate range), we can define a map $M = M_{n,m,f}$,
	\[ M : \{\l : \Cont(\l) < 0 \} \rightarrow \{ \l : \Cont(\l)\geq 0 \}, \]
so that $|\Cont(M(\l))| > d|\Cont(\l)|$ \emph{and} the preimage of each element has size at most $d$.
It shall then follow that $[z^n]f$ is positive.

To define this map, we need to define a simple combinatorial quantity. For a polynomial
$f = \sum_{k} a_kz^k $ we set $S^+(f) := \{ k : a_k > 0 \}$ and $S^-(f) := \{ k : a_k <0\}$, as before. If $f$ has the positive covering property
then for every $k \in S^-(f)$ we may express $k = b_1 + \cdots + b_{t}$, for some $t\in \N$ and $b_1,\ldots,b_t \in S^+(f)\setminus \{0\}$.
Let us define the \emph{weight} of $k$ to be the maximum $t$ for which this holds. That is, 
\[ w_f(k) := \max\{  t \in \N:  k = b_1 + \cdots + b_t \textit{ where } b_i \in S^+(f)\setminus \{0\} \}. \]
We then define 
\[ w(f) = \min_{k \in S^{-}(f)} w_f(k).\]

To define the map $M$, let $f$ be a polynomial with the positive covering property and let $m,n \in \N$.
For each $\l \in  \{\l : \Cont_{n,m}(\l) < 0 \}$, there must be some minimal $j$ for which $a_j <0$ and $\l_j >0$.
By the covering property, we know that there exists a partition $\mu  = (1^{\mu_1},2^{\mu_2}, \ldots, d^{\mu_d} )$ of $j$ 
for which $\mu_i > 0$ implies $a_j >0$ for each $j \in \{0,1,\ldots,d\}$  and $\sum_{i} \mu_i \geq w(f)$. Then define a partition $\tilde{\l}$ of $n$ by 
$\tilde{\l}_j = \l_j -1$ and $\tilde{\l}_i = \l_i+\mu_i$, for $i \not= j$. Finally set $M(\l) := \tilde{\l}$.

We pause to make a simple observation about this map.

\begin{observation}\label{obs:compression} Let $m,n,f, \l,\tilde{\l}, j$ be as above. Then, for each $i \in \{1,\ldots,d\}$, we have
\[ \left| \l_1 + \cdots + \l_i - \left( \tilde{\l}_1 +\cdots + \tilde{\l}_i\right) \right| \leq j. \]
\end{observation}

\begin{lemma}\label{lem:combheart} Let $f \in \R[z]$ be a polynomial of degree $d$ with the positive covering property and let $w := w(f)$ be as defined above. 
For all sufficiently large $m \in \N$ and $n \in  [0, m^{1-3/(2w)} ]$ let $\lambda \vdash n$ be a partition of $n$ for which \[ \Cont(\l) = \Cont_{n,m,f}(\l) < 0\] then 
\[ \Cont(M(\l)) > d |\Cont(\l)|, \]where $M$ is the map defined above.  \end{lemma}

\begin{proof}
First observe that $\Cont(\l) < 0$ implies $\Cont(M(\l)) > 0$, by construction. Hence we only need to show that $|\Cont(M(\l))| > d|\Cont(\l)|$. 
For this, we let $j$ and $\tilde{\l}_j$ be as in the definition of $M(\l)$ and we write
\[ \Cont(\l) =\binom{m}{\l_1} \binom{m-\l_1}{\l_2} \cdots \binom{m-\l_1 -\cdots - \l_{d-1}}{\l_d} \prod_i a_i^{\l_i}, 
 \] and similarly for $\Cont(\tilde{\l})$. So using that $\tilde{\l}_i = \l_i + \mu_i$, for $i \not=j$ and $\tilde{\l}_j = \l_j-1$ we have
\begin{equation}\label{eq:new-cont-ratio}
\frac{\Cont(\lambda)}{\Cont(\tlambda)} = a_j\prod_{i=0}^d a_i^{-\mu_i} \prod_{i=1}^d 
\frac{ \binom{m - \l_1 - \dots -\l_{i-1}}{\l_i}}{\binom{m - \tilde{\l}_1 - \dots -\tilde{\l}_{i-1}}{\tilde{\l}_i}}.
\end{equation}
Now note that since $\mu$ is a partition of $j$ each $\mu_i \in [0,d]$ and therefore there is a constant $C = C(f)$ for which  
\[ |a_j|\prod_{i} |a_i|^{-\mu_i}  \leq C(f).\]
We now use that $n \leq m^{1-3/(2w)}$ and therefore $ m - \l_1 -\cdots - \l_{i-1} \geq m/2$, for $m$ large. Moreover, we have that 
$ \left| \l_1 + \cdots + \l_i - \left( \tilde{\l}_1 +\cdots + \tilde{\l}_i\right)\right| \leq j$ by Observation~\ref{obs:compression} and thus we may apply Lemma~\ref{lem:gen-binomial-bound} to obtain
\[ \frac{ \binom{m - \l_1 - \dots -\l_{i-1}}{\l_i}}{\binom{m - \tilde{\l}_1 - \dots -\tilde{\l}_{i-1}}{\l_i + \mu_i} } \leq   C_d (m/2)^{-3\mu_i/(2w)}, 
\] when $i \not= j$ and when $i=j$, we have
\[  \frac{ \binom{m - \l_1 - \dots -\l_{j-1}}{\l_j}}{\binom{m - \tilde{\l}_1 - \dots -\tilde{\l}_{j-1}}{\l_j - 1 } } \leq  m. 
\]So from \eqref{eq:new-cont-ratio} we obtain
\[ \frac{|\Cont(\lambda)|}{|\Cont(\tlambda)|} \leq C_d^d \frac{|a_j|\prod_{i=0}^d |a_i|^{-\mu_i}m}{ m^{(3/(2w))\sum_{ i }\mu_i}} \leq  C_d^d\frac{C(f)}{m^{1/2}} < 1/d, \] if $m$ is sufficiently large compared to $C_d^d \cdot C(f)$. For the penultimate inequality, we have used that $\sum_i \mu_i \geq w = w(f)$, which holds by the definition of $w(f)$.
\end{proof}

\begin{corollary}\label{cor:non-zero-small}
Let $f \in \R[z]$ be a polynomial of degree $d$ with the positive covering property and put $w := w(f)$. Then 
for all $n \in [0,m^{1-3/(2w)}]$ we have \[ [z^n]f^m \geq 0,\] for sufficiently large $m$.
\end{corollary}
\begin{proof}
Let $M = M_{f,n,m}$ denote the map defined above. We have that  
\begin{align*}
[z^n]f^m &= \sum_{\l \vdash n}  \Cont(\l) \\
 &\geq \sum_{\Cont{\l} < 0 } \Cont(\l) + d^{-1}\Cont(M(\l)) \\
 & \geq 0, 
\end{align*}
where the penultimate inequality holds due to the fact that each element of $\{ \l : \Cont(\l) \geq 0\}$ is mapped to by at most $d$ distinct partitions.
The last inequality holds by applying Lemma~\ref{lem:combheart}.
\end{proof}

Lemma \ref{lem:poly-many} follows from applying Corollary \ref{cor:non-zero-small} twice:  first we apply it to see that $w(f^m)$ gets arbitrarily large as $m$ gets large.  We then apply Corollary~\ref{cor:non-zero-small} again
to $ f^m, f^{m+1},f^{m-1}$ to prove Lemma~\ref{lem:poly-many}. For this we need a simple observation. 

\begin{observation}\label{obs:2}
Let $x,m \in \N$. If $x \geq 8m^2$ then we may write
\[ x = am + b(m+1) + c(m-1), \]
for integers $a,b,c \geq x/(4m)$.
\end{observation}
\begin{proof}
We first see that for $x = km$, $k \in \N$, we may find a solution $(a_0,b_0,c_0)$ where $a_0,b_0,c_0 \in \{\lfloor k/3 \rfloor, \lceil k/3 \rceil\}$.
Then, for $x = km + \ell$ we find a solution $(a,b,c)$ by setting $a = a_0 - \ell$ and $b = b_0+\ell$, $c = c_0$. 
\end{proof}

\begin{proof}[Proof of Lemma \ref{lem:poly-many}]
	First choose $w_0\in \N$ large enough so that $1 - 3/(2w_0) > 1-\eps/2$. 
	
	We now apply Corollary~\ref{cor:non-zero-small} to find a $m_0$ so that for all $m \geq m_0$ we have $[z^k]f^{m} \geq 0$ for all $k \in \{0,1,\ldots, dw_0\}$. 
	We claim that $w(f^{m}) \geq w_0$ for all $m \geq m_0$. To see this, let $n$ be such that $[z^n]f^{m} < 0$ and note that we must have $n > dw_0$. 
	Now let $S(f) = \{ k : [z^k]f \not=0 \}$ be the support of $f$ and since $n$ is in the support $S(f^{m})$, 
	we may write $n = b_1 + \cdots + b_{t}$, where $b_1,\ldots,b_t \in S(f)\setminus \{0\}$ and $t \geq w_0$. Since $f$ has the positive covering property 
	we can write $n = b^+_1 + \cdots + b^+_{s}$ for $b_1^+ \ldots, b_s^+ \in S^+(f)\setminus \{0\}$ for $s \geq t \geq w_0$. Therefore $w(f^{m}) \geq w_0$.
	
	Now, choose $m_1 = m_0+1$ and 
	 given $x \geq 8(m_1)^2$ we may apply Observation~\ref{obs:2} to write $x = am_1 + b(m_1+1) + c(m_1-1)$, where $a,b,c \geq x/(4m_1)$ and thus we may write
	\begin{equation} \label{eq:factorf} f^{x} = f^{am_1}f^{b(m_1+1)}f^{c(m_1-1)}\end{equation}

	Now crucially note that each of $f^{m_1},f^{m_1-1},f^{m_1+1}$ has the positive covering property and that $w(f^{m_1}),w(f^{m_1-1}),w(f^{m_1+1}) \geq w_0$.
	Therefore we may apply Lemma~\ref{cor:non-zero-small}
	to each of $ f^{m_1}, f^{m_1+1} f^{m_1-1}$ to learn that for $ q \in \{m_1,m_1-1,m_1+1\}$
	we have $[z^n]f^{p q} \geq 0$ for all $n \in [0,p^{1-2/w_0}] \supseteq [0,p^{1-\eps/2}] $ and sufficiently large $p$. 
	Thus, from \eqref{eq:factorf}, we see that $[z^n]f^x \geq 0$ for all $n \in [0,x_0^{1-\eps/2}]$ where $x_0 = \min\{ a,b,c\} \geq x/(4m_1)$
	and $x$ is sufficiently large. This completes the proof.	
\end{proof}

\section{Dealing with $n \in [0,\delta m]$ and the proof of the Theorem~\ref{thm:main}}\label{sec:proof}

In this section we bolster the main result of the previous section (Lemma~\ref{lem:poly-many}) by showing that we can get positive coefficients for all $n \in [0,\delta m]$, for sufficiently large $m$.

\begin{lemma} \label{lem:first-delta}
	Let $f(z) = 1 + a_1 z + \cdots + a_d z^d$ have the positive covering property and assume that $f\not= g(z^{\ell})$ for $g \in \R[z]$, $\ell \geq 2$.  Then there exists a $\delta > 0$ so that for all sufficiently large $m$ we have
 $[z^n] f^m \geq 0$ for all $n \in [0,\delta m]$.  
\end{lemma}

One we have proved Lemma \ref{lem:first-delta} we essentially be finished, by appealing to the following result of De Angelis \cite{deAngelis-positivity}.

\begin{lemma}[\cite{deAngelis-positivity}]\label{lem:middle-coeffs}
	Let $f(z) = a_0 + \cdots + a_d z^d$ be strongly positive with $a_0, a_d > 0$.  Then for every $\delta > 0$, there exists an $M = M(\delta)$ so that for all $m \geq M$, $[z^n] f(z)^m > 0$ for all $n \in [\delta m, (d - \delta)m]$.
\end{lemma}

This lemma is an immediate consequence of \cite[Theorem 4.1]{deAngelis-positivity}, which provides an asymptotic expansion for the coefficients $[z^n] f(z)^m$, for $n/m$ lying in a compact subset of $(0,\infty)$, under the assumption of strong positivity.  We note that De Angelis uses the notation $\sim$ to denote an \emph{asymptotic expansion} which in particular means that the first term of the sum is larger then the remainder; since the first term of the expansion obtained by De Angelis is positive, this shows our Lemma \ref{lem:middle-coeffs}.

For the proof of Lemma \ref{lem:first-delta} we will again only require the positive covering property, and not strong-positivity.  The proof is entirely analytic, and makes use of the \emph{saddle-point method}: the coefficient $[z^n] f(z)^m$ is written as a contour integral using Cauchy's integral formula; by a careful choice of the contour, we may minimize the oscillation in the integrand and show that the integral is dominated by the piece of the contour nearest to the positive real-axis.  This method was used by Bergweiler, Eremenko and Sokal \cite{bes} in their proof of Theorem \ref{th:deAngelis} and is similar to De Angelis' proof \cite{deAngelis-positivity}. Our new ingredient starts with the observation that the positive covering property implies a quantitative version of strong positivity for $z$ in a neighborhood of the origin. 
This observation is captured in the following lemma. 
 
\begin{lemma}\label{lem:quant-positive}
	Let $f(z) = 1 + a_1 z + \cdots + a_{d-1}z^{d-1} + z^d$ satisfy the positive covering property and assume that $f$ is not of the form $g(z^{\ell})$, for $g \in \R[z]$, $\ell \geq 2$.
	If $\theta_0 \in (0,\pi)$ then there is a constant $c = c(f,\t_0) >0$  so that for all $|\theta| \geq \theta_0$  we have 
	\begin{equation} \label{eq:state-quant-positive}  |f(r\eit)| \leq (1 - cr^d)|f(r)|  \, , \end{equation}
	for sufficiently small $r$.
\end{lemma}

\begin{proof}
By compactness of the set $[\theta_0,\pi]$, it is sufficient to show that \eqref{eq:state-quant-positive} holds for a neighborhood of each $\phi \in [\theta_0,\pi]$ with a constant which may depend on $\phi$. 

Fix $\phi \in [\theta_0,\pi]$ and define $T := \{ j \geq 1 : a_j \neq 0\ \mathrm{ and }\ e^{ij\phi} \neq 1 \}, $ and $T^{\ast} := \{ j \geq 1 : a_j \neq 0\ \mathrm{ and }\ e^{ij\phi} = 1 \}$ so that $T\cup T^{\ast} = S(f) \setminus \{0\} = \{ k \geq 1 : a_k \not= 0\}$. Note that if $T = \emptyset$ then we must have 
$T^{\ast} \not= \emptyset$ and therefore $2\pi/\phi \in \Z$. Thus $f(z) = g(z^{\ell})$ for $\ell = 2\pi/\phi \in \Z$, contradicting the condition on $f$.  So we can assume $T \not= \emptyset$ and thus we may set $t_0 := \min T$. There are two cases: where $t_0 < t$ for all $t \in T^\ast$ and where $\min T^\ast < t_0$.

In either case, we expand  
\begin{equation} \label{eq:fdiff-exp} 
|f(r)|^2 - |f(r \eit)|^2= \sum_{j\geq 1} r^j \left( \sum_{p+q=j} a_pa_q ( 1 - \cos((p-q)\t)) \right)
\end{equation}
where the inner sum is over all ordered pairs $(p,q)$ of non-negative integers summing to $j$.

\begin{case}\label{case:first}
	$t_0 < t$ for all $t \in T^\ast$.
\end{case}
Since $T \cup T^\ast = S(f) \setminus \{0\}$, we have that $f(z) = 1 + a_{t_0} z^{t_0} + o(|z|^{t_0})$ as $|z| \to 0$; the positive covering hypothesis then implies $a_{t_0} > 0$.  The expansion \eqref{eq:fdiff-exp} can be written  
$$
|f(r)|^2 - |f(r \eit)|^2= 2 a_{t_0} r^{t_0} (1 - \cos(t_{0} \theta)) + o(r^{t_0})\,.
$$
 
Since $t_0 \in T$, the term $1 - \cos(t_0\theta)$ does not vanish for $\theta$ in some neighborhood of $\phi$.  In particular, there is some neighborhood on which $1 - \cos(t_0 \theta) \geq c$ for some $c > 0$.  For $r$ sufficiently small, the term $c a_{t_0} r^{t_0}$ dominates the $o(r^{t_0})$ term, and so there is the bound 
$$|f(r)|^2 - |f(r \eit)|^2 \geq \frac{1}{2} c a_{t_0} r^{t_0} $$
for $\theta$ in some neighborhood of $\phi$ and $r$ sufficiently small. Dividing by $f(r)^2$ and utilizing $f(r) \sim 1$ as $r \to 0$ completes the lemma in this case.
 
\begin{case} $\min T^\ast < t$.
 \end{case}

The expansion \eqref{eq:fdiff-exp} will again be used.  The idea will be to break the summands into terms that vanish at $\theta = \phi$ and those that do not; the lowest degree terms of each type will have positive coefficients and dominate the other terms in their respective sum for small $r$.  To begin, let $t_\ast := \min T^\ast$.  The positive covering hypothesis implies that $a_{t_\ast} > 0$ since $a_{t_\ast} z^{t_\ast}$ is the lowest degree term in $f$.  We now observe that, the positive covering hypothesis implies that $a_{t_0} > 0$ as well.

\begin{claim}\label{claim:t0-positive}
	 $a_{t_0} > 0$.
\end{claim} 
\begin{proof}[{Proof of Claim \ref{claim:t0-positive}}] For a contradiction, assume that $a_{t_0} < 0$.  The positive covering hypothesis implies that there exist $j_1,\ldots,j_k$ with $j_1 + \cdots + j_k = t_0$ and each $a_{j_i} > 0$.  By minimality of $t_0$, it must be the case that each $j_i \in T^\ast$.  However, this would imply $$e^{i t_{0} \phi} = e^{i j_1 \phi} \cdots e^{i j_k \phi} = 1$$ contradicting the assumption that $t_0 \in T$.  \end{proof}

We now group the terms that appear in \eqref{eq:fdiff-exp}. We define
\[ \mathsf{V} := \{(p,q): a_pa_q \neq 0,\cos((p-q)\phi) = 1  \}\] and 
\[ \mathsf{NV}:= \{(p,q) : a_pa_q \neq 0 , \cos((p-q)\phi) \neq 1\}.
\] And so, from \eqref{eq:fdiff-exp}, we may express $|f(r)|^2 - |f(r\eit)|^2$ as 

\begin{equation} \label{eq:fdiff-exp-case2} \sum_{(p,q) \in \mathsf{V}} r^{p+q} a_p a_q(1 - \cos((p-q)\theta) )
+ \sum_{(p,q) \in \mathsf{NV}} r^{p+q} a_p a_q(1 - \cos((p-q)\theta) )\,. 
\end{equation}

Each of the two sums in \eqref{eq:fdiff-exp-case2} will be lower bounded individually.  Beginning with the sum over $\mathsf{NV}$,
we note that if both $p$ and $q$ are in $T^\ast$, then $\cos((p-q)\phi) = 1$.  This implies that the lowest power of $r$ appearing in the sum over $\mathsf{NV}$ is $r^{t_0}$; in particular, we have the expansion \begin{equation} \label{eq:non-vanishing}
\sum_{(p,q) \in \mathsf{NV}} r^{p+q} a_p a_q(1 - \cos((p-q)\theta) ) = 2 a_{t_0} r^{t_0} (1 - \cos(t_0 \theta)) + o(r^{t_0})\,.
\end{equation}

As in Case \ref{case:first}, this may be bounded below by $c r^{t_0}$ for $\theta$ in a neighborhood of $\phi$ and $r$ sufficiently small.  

We treat the sum over pairs in $\mathsf{V}$ in a similar way. We express the sum 
\[ \sum_{(p,q) \in \mathsf{V}} r^{p+q} a_p a_q(1 - \cos((p-q)\theta) ) \]  as
\begin{equation}\label{eq:vanishing} 2 r^{t_\ast} a_{t_\ast} (1 - \cos(t_\ast \theta)) + \sum_{(p,q) \in \mathsf{V}: p+q > t_\ast} r^{p+q} a_p a_q(1 - \cos((p-q)\theta) )\, ,\end{equation}
where we have isolated the dominant term. Recall that $a_{t_\ast} >0$.

Since $\cos(t_\ast \phi) = 1$ and $\cos((p-q)\phi) = 1$ for each $(p,q) \in \mathsf{V}$, we may take $\theta$ in a neighborhood of $\phi$ so that we simultaneously have $$1 - \cos(t_\ast \theta) \geq \frac{1}{4}(t_\ast(\theta - \phi))^2\quad\mathrm{and}\quad 1 - \cos((p-q) \theta) \leq ((p-q)(\theta - \phi))^2\,.$$

Thus, for all $\t$ in this neighbourhood we have \begin{equation} \label{eq:vanish-final}
\sum_{(p,q) \in \mathsf{V}} r^{p+q} a_p a_q(1 - \cos((p-q)\theta) ) \geq r^{t_\ast}(\theta - \phi)^2\left(\frac{t_\ast^2}{2}a_{t_\ast} + o(1) \right) \geq 0\,,
\end{equation} as $r \rightarrow 0$.
Plugging lines \eqref{eq:non-vanishing} and \eqref{eq:vanish-final} into \eqref{eq:fdiff-exp-case2} completes the proof.
\end{proof}

\vspace{4mm}

We may now prove Lemma~\ref{lem:first-delta}.

\begin{proof}[Proof of Lemma \ref{lem:first-delta}]
	Define $k$ to be the minimum $\ell \geq 1$ for which $a_{\ell} \not= 0 $.  Thus we may write 
	\[f(z) = 1 + a_k z^k + \cdots + a_d z^d.\] We may also assume that $a_k = \frac{1}{k}$, by possibly rescaling $z$; i.e. $z \mapsto \lambda z$.
	We now write $z = \rho e^{i\theta}$ and use Cauchy's integral formula to express  \begin{equation} \label{eq:cauchy-int}
	[z^n]f^m = \frac{1}{2\pi i} \int_{|z|= \rho} \frac{f(z)^m}{z^n} \, \frac{dz}{z} =  \frac{1}{2\pi}\int_{-\pi}^\pi \frac{f(\rho e^{i\theta})^m}{(\rho e^{i\theta})^n } \, d\theta\,.
	\end{equation}	
	We will find asymptotics as $m \to \infty$ of the right hand of this equation for an appropriate choice of $\rho$. For this, set $\alpha := n/m$, and note that by Lemma \ref{lem:poly-many} we know that $[z^n]f(z)^{m} \geq 0$ for all $\alpha \leq  m^{-1/(d+1)}$ and sufficiently large $m$.  Thus, it is sufficient to assume $\alpha \geq  m^{-{1}/({d+1})}$.  
	Define \[h_\rho(\theta) := \log f(\rho \eit) - \alpha \log(\rho \eit) \] 
	so that 
	$$ \frac{f(\rho \eit)^m}{(\rho \eit)^n} = \exp\left(m \left(\log f(\rho \eit) - \alpha \log (\rho \eit) \right) \right) = \exp\left( m h_\rho(\theta) \right)\,.$$
	
	We now make a careful choice of $\rho$, which will make the integral at \eqref{eq:cauchy-int} easier to calculate: we choose $\rho$ so that $h_\rho'(0) = 0$.  
	To see that such a $\rho$ exists, we express
	$$h_\rho'(\theta) = i\left( \rho \eit \frac{f'(\rho \eit)}{f(\rho \eit)} - \alpha \right)$$ 
	and thus, as $\rho \to 0$, we have 
	$$h_\rho'(0) = i\left(\rho \frac{f'(\rho)}{f(\rho)} - \alpha \right) = i\left(\rho^k + O(\rho^{k+1}) - \alpha \right)\, ,$$ implying that for $\alpha$ sufficiently small, there exists a $\rho$ for which $h_\rho'(0) = 0$.  We also see that for this solution $\rho$, we have\begin{equation}\label{eq:rho-alpha} 
	\rho^k \sim \alpha\,.
	\end{equation}
	To proceed further, we will need control over the second and third derivative. We compute 
	\begin{align*}
	h_\rho''(\theta) &= - \left( z\frac{f'(z)}{f(z)} + z^2 \frac{f''(z)}{f(z)} -  \left(z \frac{f'(z)}{f(z)}\right)^2 \right) \\
	h_\rho'''(\theta) &= -i \left[ z\frac{f'(z)}{f(z)} + z^2 \frac{f''(z)}{f(z)} -  \left(z \frac{f'(z)}{f(z)}\right)^2+ 2 z^2\left(\frac{f''(z)}{f(z)} - \left(\frac{f'(z)}{f(z)} \right)^2  \right) \right. \\
	& \left. \qquad+ z^3 \left(\frac{f'''(z)}{f(z)} - 3\frac{f''(z)f'(z)}{f(z)^2}  + 2 \left(\frac{f'(z)}{f(z)} \right)^3 \right)   \right] \,.
	\end{align*}
	
	We care particularly about the behavior of $h_\rho''(\theta)$ and $h_\rho'''(\theta)$ for $\rho$ near zero, so we note \begin{equation} \label{eq:h-derivs-small}
	h_\rho''(\theta) \sim - k z^k,\qquad h_\rho'''(\theta) \sim -i k^2 z^k\qquad \mathrm{as}\ z \to 0\,.
	\end{equation}	
	
	By \eqref{eq:rho-alpha}, taking $\alpha \to 0$ implies $\rho \to 0$.  Therefore, by \eqref{eq:h-derivs-small}, we may take $\alpha$ sufficiently small so that each of 
	the quantities
	\begin{equation}\label{eq:h-derivs-bounds} 
	 \frac{\rho^k}{\alpha},\ \ \frac{|h_\rho''(0)|}{k \rho^k}, \ \ \sup_{\theta} \frac{|h_\rho'''(\theta)|}{k^2 \rho^k}\,
	\end{equation} lie in the interval $[3/4,5/4]$.

We now divide the integral in \eqref{eq:cauchy-int} into three pieces. For this let $\theta_0 := 1/k$ and $\eta := \eta(\rho) = (k^2 m \rho^k)^{-1/3}$. 
We then define 
	\begin{equation}
	I_1 = \int_{|\theta| < \eta} e^{m h_\rho(\theta)}\,d\theta,\quad I_2 = \int_{|\theta| \in (\eta,\theta_0)} e^{m h_\rho(\theta)}\,d\theta,\quad I_3 = \int_{|\theta| > \theta_0} e^{m h_\rho(\theta)}\,d\theta\,.
	\end{equation}  Since $f$ has real coefficients, each of these three integrals is real. 
To estimate each of these integrals, we use a Taylor expansion of $h_{\rho}(\t)$ at $\t = 0$. Indeed,
	 \begin{equation}\label{eq:h-taylor} h_\rho(\theta) = h_\rho(0) + \frac{1}{2}h_\rho''(0) \theta^2 + R(\theta) , \end{equation}
	where \begin{equation} \label{eq:h-remainer} |R(\theta)| \leq \frac{1}{6} \sup_{\theta} |h_\rho'''(\theta)| \cdot |\theta|^3 \leq  \frac{5 k^2 }{24} \rho^k |\theta^3|\,. \end{equation} Our first claim shows that $I_1$ is large and positive. 

\begin{claim} \label{claim:I1}
\[ I_1 \geq \frac{C_fe^{mh_{\rho}(0)}}{\sqrt{m\rho^k}}  , \]
where $C_f>0$ is a constant depending only on $f$. 
\end{claim}	
\begin{proof}[{Proof of Claim~\ref{claim:I1}}] Since $\overline{h}_\rho(\theta) = h_\rho(-\theta)$, 
we have \[ I_1 = 2 \int_0^\eta \Re (e^{m h_\rho(\theta)})\,d\theta\,.
\] To obtain a lower bound on the real part of the integrand, we obtain an upper bound $\Im (m R(\theta))$ for $\theta \in (0, \eta)$ by using \eqref{eq:h-remainer}, to get $$\Im(m R(\theta)) \leq \frac{5 k^2}{24}m \rho^k(k^2 m \rho^k)^{-1} = \frac{5}{24}\,.$$ 
	
	Now, using the expansion at \eqref{eq:h-taylor} and the choice of $\eta= (k^2 m \rho^k)^{-1/3}$, we obtain a lower bound\begin{align*}
	I_1  &= 2 \int_0^\eta \Re (e^{m h_\rho(\theta)})\,d\theta \\
	&\geq c_1 e^{m h_\rho(0)} \int_0^\eta \exp\left({m \frac{h_\rho''(0)}{2}\theta^2 } \right)\,d\theta \\
	&\geq c_1 e^{m h_\rho(0)} \int_0^{(k^2 m \rho^k)^{-1/3}} \exp\left(- \frac{5k}{8} m \rho^k \theta^2 \right)\,d\theta \\
	&= \frac{c_2}{\sqrt{k m \rho^k }} e^{mh_\rho(0)} \int_0^{c_3( m \rho^k / k )^{1/6}} e^{-x^2} \,dx \\
	&\geq \frac{C_f}{\sqrt{m \rho^k}} e^{m h_\rho(0)},
	\end{align*}
	where we have used $m \rho^k \to \infty$ and set 
	\[ C_f := \frac{c_2}{\sqrt{k}} \int_0^1 e^{-x^2}\,dx >0,\]
	which is a constant, depending only on $f$. This completes the proof of the claim. 
\end{proof}
	
The next two claims show that $I_2$ and $I_3$ are small. 	

\begin{claim}\label{claim:I2}
\[|I_2| \leq e^{m h_{\rho}(0)} e^{-c_2(m\rho^k)^{1/3}}, \]
Where $c_2>0$ is a constant depending only on $f$.
\end{claim}
\begin{proof}[Proof of Claim~\ref{claim:I2}]  Note that for $|\theta| \leq \frac{1}{k}$, we have 
	$$\Re\left(\frac{h''(0)}{2} \theta^2 + R(\theta)\right) \leq -\frac{3k}{8}\rho^k \theta^2 + \frac{5}{24} k^2 \rho^k|\theta^3| = k \rho^k \theta^2 \left(-\frac{3}{8} + \frac{5}{24}k|\theta| \right) < - \frac{k \rho^k \theta^2}{6}\,. $$
	
Thus, using \eqref{eq:h-taylor} and the triangle inequality, we have $$\left|I_2 \right| \leq e^{mh_\rho(0)}\int_{|\theta| \in (\eta,1/k)}  e^{-\frac{k}{6} m \rho^k \theta^2 }\,d\theta \leq \frac{2}{k} e^{m h_\rho(0)} e^{-\frac{k^{-1/3}}{6}(m\rho^k)^{1/3}}\,.$$\end{proof}

	\begin{claim}\label{claim:I3}
	$$|I_3| \leq e^{m h_\rho(0)} \exp\left(-c_3 m^{\frac{1}{d+1}} \right)\,, $$ where $c_3 >0$ is an absolute constant.
	\end{claim}
\begin{proof}[Proof of Claim~\ref{claim:I3}]
	To bound $I_3$, we can apply Lemma \ref{lem:quant-positive} with $\theta_0 = \frac{1}{k}$ to get a constant $c > 0$ so that for $|\theta| > \frac{1}{k}$ we have $$\frac{|f(\rho \eit)|^m}{\rho^n} \leq \frac{f(\rho)^m}{\rho^n}(1 - c \rho^d)^m = e^{mh_\rho(0)} (1 - c\rho^d)^m\,.$$
	
	Using that $\rho^k \geq c\alpha$ and therefore, $\rho^d \geq c\alpha^{d/k} \geq c\alpha^{d}$, we have
	
	$$(1 - c \rho^d)^m \leq \exp\left(- c m\rho^d \right) \leq \exp\left(- c_3 m \alpha^d \right) \leq \exp\left(- c_3 m^{\frac{1}{d+1}} \right)\,, $$
	where the last inequality follows from our assumption that $\alpha \geq m^{-1/(d+1)}$.
	
	This gives an upper bound on $$|I_3| \leq 2\pi e^{m h_\rho(0)} \exp\left(-c_3 m^{\frac{1}{d+1}} \right)\,.$$ This completes the proof of Claim~\ref{claim:I3}.
\end{proof}
	
\vspace{4mm}

We now simply apply Claims~\ref{claim:I1}, \ref{claim:I2} and \ref{claim:I3} to estimate \eqref{eq:cauchy-int} and finish. Indeed, we have \begin{align*} 2\pi [z^n]f(z)^m &\geq I_1 - |I_2| - |I_3| \\
	&\geq e^{m h_\rho(0)} \left(\frac{C_f}{\sqrt{m\rho^k}}  - e^{-c_2(m\rho^k)^{1/3}} -  e^{-c_3 m^{\frac{1}{d+1}}} \right) \\
	&> 0, \end{align*} for sufficiently large $m$. This completes the proof of Lemma~\ref{lem:first-delta}.\end{proof}

\vspace{4mm}

Before diving into the ``hard'' direction of our equivalence, we quickly spell out details of the ``easy'' direction.

\begin{lemma}\label{lem:easy-direction}
If $f \in \R[x]$ is non-constant and eventually non-negative then $f(x) = x^kg(x^{\ell})$, where $g$ is strongly positive and satisfies the positive covering property.
\end{lemma}
\begin{proof}We may write $f(z) = z^kg(z^{\ell})$, with $g(z) = b_0 + b_1z + \cdots +b_dz^d$, where $b_0,b_d \not=0$ and $g(z)$ is not equal to any polynomial of the form 
$z^{k'}h(z^{\ell'})$, for $k' \geq 0$ and $\ell'\geq 2$. Also note that since $f$ is eventually non-negative, we must have that $b_0 > 0$. We also notice that $f$ is eventually non-negative if and only if $g$ is. 
We first show that $g$ is strongly positive, as we sketched in the introduction. 
If $g$ is eventually non negative, then we may choose an odd $m$ for which $g^m$ has non-negative coefficients. We write $g^m(z) = \sum_{k}b^{(m)}_kz^k $, let $z = re^{i\t} \in \C \setminus \R_{\geq 0}$ and notice that 
\[ |g^m(z)| = \left| \sum_{k} b^{(m)}_k \rho^ke^{ik\t} \right| \leq \sum_{k} b^{(m)}_k \rho^k = g^m(|z|), \]
by the triangle inequality and therefore $|g(z)| \leq g(|z|)$, since $m$ is odd. In fact, we can see that this must be a strict inequality, for if 
\[ |g(z)| = \left| \sum_{k} b_k\rho^ke^{i\t k}  \right| =   \sum_{k} b_k\rho^k  = g(|z|)\] 
we would have $e^{ik_1\t} = e^{ik_2\t} $ for any $k_1,k_2$ in the support of the above sum. Thus $e^{i\t}$ is a $k_1-k_2$ root of unity for any such $k_1,k_2$. This would imply that the support of $g$ is contained in a proper arithmetic progression of $\Z$, implying that $g(z) = z^{k'}h(z^{\ell'})$ for some $h \in \R[z]$ and 
$k' \in \Z_{\geq 0}, \ell' \geq 2$ which is a contradiction. Thus $g$ is strongly positive.

We now show that $g$ has the positive covering property. Assume, without loss of generality, that $g$ fails the one-sided covering property (otherwise replace $g$ with its ``reverse'' $z^dg(1/z)$). Therefore, there exists some minimal $k$ for which $b_k < 0 $ but $k$ is not contained in the $d$-fold sum $S^+(g) + \cdots + S^+(g)$.  Let us write  
\[ k = i_1 + \cdots + i_m ,\]
for some $i_1,\ldots,i_m \in S(g) := \{ i : b_i\not= 0 \}$. We shall show that all but one of $i_1,\ldots, i_m$ are $0$ and one is $k$.
Now, note that we cannot have $i_1,\ldots,i_m \in S^+(g)$, as this would imply that $k \leq d$ can be written as a $d$-fold sum of elements in $S^+(g)$, therefore some $i_j$ must be such that 
$b_{i_j} < 0$. 
Without loss of generality, assume that these terms are exactly $i_1 \geq \cdots \geq i_s$. If $i_1 < k$, by the minimal choice of $k$, for $j \in [s]$ we have, $i_j = r_1^{(j)} +\cdots + r^{(j)}_d$ where $r^{(j)}_1,\ldots,r^{(j)}_d \in S^+(g)$ and therefore 
\[ k = (r^{(1)}_1 + \cdots + r^{(1)}_d) + (r^{(2)}_1 + \cdots + r^{(2)}_d) + \cdots (r^{(s)}_1 + \cdots + r^{(s)}_d)+ i_{s+1} + \cdots +  i_m.\] Note that since each of these terms is non-negative, it must be that $k$ is in the $d$ fold sum 
$S^+(g) + \cdots + S^+(g)$, a contradiction. Hence it must be that $i_1 = k$ and $i_2,\ldots, i_d = 0$. 
Now, if we let $b^{(m)}_k$ be as above, we see
\[ b^{(m)}_k = \sum_{i_1 + \cdots + i_m = k} b_{i_1} \cdots b_{i_m} = m (b_0^{m-1}b_k) < 0, \]
for all $m$, which contradicts the assumption that $f$ is eventually non-negative.
\end{proof}

\vspace{4mm}

The proof of Theorem~\ref{thm:main} now only requires us to evoke our lemmas.

\begin{proof}[Proof of Theorem~\ref{thm:main}] 
	We write $f(z) = z^kg(z^{\ell})$ and assume that $g$ cannot be expressed in this form. 
	
	Now $f$ is eventually non-negative if and only if $g$ is. By Lemma~\ref{lem:easy-direction}, we see that if $g$ is eventually non-negative, then $g$ is strongly positive and satisfies the positive covering property.
	
	To see the converse, assume that $d = \deg(g)$ and apply Lemma~\ref{lem:first-delta} to  $g(z)$ and $z^dg(1/z)$
	to obtain a $\delta = \delta(f)>0$ and to learn that for all 
	$n \in [0,\delta m] \cup [(1-\delta)dm,dm]$ we have $[z^n]f \geq 0$, for $m$ sufficiently large. We then finish by applying Lemma~\ref{lem:middle-coeffs}.
\end{proof}


\section*{Acknowledgments} We would like to thank the anonymous referees for their thorough reading of the paper. 

\bibliographystyle{amsplain}

\begin{thebibliography}{10}

\bibitem{bergweiler-eremenko}
W.~Bergweiler and A.~Eremenko.
\newblock Distribution of zeros of polynomials with positive coefficients.
\newblock {\em Ann. Acad. Sci. Fenn. Math.}, 40(1):375--383, 2015.

\bibitem{bes}
W.~Bergweiler, A.~Eremenko, and A.~Sokal.
\newblock Roots of polynomials with positive coefficients.
\newblock
  \url{http://citeseerx.ist.psu.edu/viewdoc/summary?doi=10.1.1.381.145}, 2013.

\bibitem{BBL}
J.~Borcea, P.~Br\"{a}nd\'{e}n, and T.~M. Liggett.
\newblock Negative dependence and the geometry of polynomials.
\newblock {\em J. Amer. Math. Soc.}, 22(2):521--567, 2009.

\bibitem{deAngelis-positivity}
V.~De~Angelis.
\newblock Asymptotic expansions and positivity of coefficients for large powers
  of analytic functions.
\newblock {\em Int. J. Math. Math. Sci.}, (16):1003--1025, 2003.

\bibitem{handelman-85}
D.~Handelman.
\newblock Positive polynomials and product type actions of compact groups.
\newblock {\em Mem. Amer. Math. Soc.}, 54(320):xi+79, 1985.

\bibitem{handelman}
D.~Handelman.
\newblock Deciding eventual positivity of polynomials.
\newblock {\em Ergodic theory and dynamical systems}, 6(1):57--79, 1986.

\bibitem{handelman1992polynomials}
D.~Handelman.
\newblock Polynomials with a positive power.
\newblock {\em Symbolic dynamics and its applications}, 135:229--230, 1992.

\bibitem{LPRS}
J.~L. Lebowitz, B.~Pittel, D.~Ruelle, and E.~R. Speer.
\newblock Central limit theorems, {L}ee-{Y}ang zeros, and graph-counting
  polynomials.
\newblock {\em J. Combin. Theory Ser. A}, 141:147--183, 2016.

\bibitem{lee-yang}
T.~D. Lee and C.~N. Yang.
\newblock Statistical theory of equations of state and phase transitions. {II}.
  {L}attice gas and {I}sing model.
\newblock {\em Phys. Rev. (2)}, 87:410--419, 1952.

\bibitem{clt2}
M.~Michelen and J.~Sahasrabudhe.
\newblock Central limit theorems and the geometry of polynomials.
\newblock {\em arXiv preprint arXiv:1908.09020}, 2019.

\bibitem{clt1}
M.~Michelen and J.~Sahasrabudhe.
\newblock Central limit theorems from the roots of probability generating
  functions.
\newblock {\em Advances in Mathematics}, 358:106840, 2019.

\bibitem{pemantle-survey}
R.~Pemantle.
\newblock Hyperbolicity and stable polynomials in combinatorics and
  probability.
\newblock In {\em Current developments in mathematics, 2011}, pages 57--123.
  Int. Press, Somerville, MA, 2012.

\bibitem{poincare}
H.~Poincar\'e.
\newblock Sur les \'equations alg\'ebriques.
\newblock {\em C.R. Acad. Sci. Paris}, 97:1418, 1883.

\bibitem{polya1928positive}
G.~P{\'o}lya.
\newblock {\"U}ber positive darstellung von polynomen.
\newblock {\em Vierteljschr. Naturforsch. Ges. Z{\"u}rich}, 73:141--145, 1928.

\bibitem{scott-sokal}
A.~D. Scott and A.~D. Sokal.
\newblock Complete monotonicity for inverse powers of some combinatorially
  defined polynomials.
\newblock {\em Acta Math.}, 213(2):323--392, 2014.

\bibitem{tan-to-2017}
C.~Tan and W.-K. To.
\newblock Characterization of polynomials whose large powers have all positive
  coefficients.
\newblock {\em Proc. Amer. Math. Soc.}, 146(2):589--600, 2018.

\bibitem{tan-to-2019}
C.~Tan and W.-K. To.
\newblock Characterization of polynomials whose large powers have fully
  positive coefficients.
\newblock {\em arXiv preprint arXiv:1902.03379}, 2019.

\bibitem{yang-lee}
C.~N. Yang and T.~D. Lee.
\newblock Statistical theory of equations of state and phase transitions. {I}.
  {T}heory of condensation.
\newblock {\em Phys. Rev. (2)}, 87:404--409, 1952.

\end{thebibliography}


\begin{dajauthors}
\begin{authorinfo}[marcus]
  Marcus Michelen\\
  University of Illinois at Chicago\\
  Chicago, Illinois, United States of America\\
  michelen.math\imageat{}gmail\imagedot{}com \\
  \url{https://marcusmichelen.org/}
\end{authorinfo}
\begin{authorinfo}[julian]
  Julian Sahasrabudhe\\
  University of Cambridge \\
  Cambridge, United Kingdom \\
  jdrs2\imageat{}cam\imagedot{}ac\imagedot{}uk \\
  \url{https://www.dpmms.cam.ac.uk/~jdrs2}
\end{authorinfo}

\end{dajauthors}

\end{document}